\documentclass[12pt]{amsart}
\usepackage[colorlinks,
linkcolor=blue,
anchorcolor=blue,
citecolor=blue]{hyperref}

\usepackage{graphicx}
\usepackage{amssymb}
\usepackage{setspace}
\usepackage{datetime}
\usepackage{amsmath}
\newtheorem{thm}{Theorem}[section]

\newtheorem{lem}[thm]{Lemma}
\newtheorem{prop}[thm]{Proposition}
\theoremstyle{definition}

\theoremstyle{remark}

\theoremstyle{conclusion}

\theoremstyle{question}
\numberwithin{equation}{section}

\begin{document}

\title[Fractional H\'enon-Lane-Emden Systems]{Liouville type theorems for fractional and higher-order fractional systems}

\author{Daomin Cao}
\address{School of Mathematics and Information Science, Guangzhou University, Guangzhou 510405, and People’s Republic of China and  Institute of Applied Mathematics, AMSS, Chinese Academy of Sciences, Beijing 100190,
	People’s Republic of China}
\email{dmcao@amt.ac.cn}

\author{Guolin Qin}
\address{Institute of Applied Mathematics, Chinese Academy of Sciences, Beijing 100190, and University of Chinese Academy of sciences, Beijing 100049, P. R. China}
\email{qinguolin18@mails.ucas.ac.cn}
\thanks{D. Cao was supported by NNSF of China (No.11831009) and Chinese Academy of Sciences (No.QYZDJ-SSW-SYS021)}

\begin{abstract}
In this paper, we first establish decay estimates for the fractional and higher-order fractional  H\'enon-Lane-Emden systems by using a nonlocal average and integral estimates,  which deduce a result of non-existence. Next, we apply the method of scaling spheres introduced in \cite{DQ2} to derive a Liouville type theorem. We also construct an interesting example on super $\frac{\alpha}{2}$-harmonic functions (Proposition 1.2).
\end{abstract}

\maketitle{\small {\bf Keywords:} Fractional Laplacians; H\'enon-Lane-Emden systems; Decay estimates; Higher-order fractional Laplacians; Liouville type theorems.\\
	

\section{Introduction}
In this paper, we are concerned with the following fractional and higher-order fractional H\'enon-Lane-Emden systems
\begin{equation}\label{HPDE}\\\begin{cases}
(-\Delta)^{k+\frac{\alpha}{2}}u(x)=|x|^av^p, \,\,\,\,\,\,\,\, x\in\mathbb{R}^{n}, \\
(-\Delta)^{l+\frac{\beta}{2}}v(x)=|x|^bu^q, \,\,\,\,\, x\in\mathbb{R}^{n},\\
u(x)\geq 0, v(x)\geq 0,\,\,\, x\in \mathbb{R}^{n},
\end{cases}\end{equation}
where $u\in C^{2k+[\alpha],\{\alpha\}+\epsilon}_{loc}(\mathbb{R}^{n})\cap\mathcal{L}_{\alpha}(\mathbb{R}^{n}),\,v\in C^{2l+[\alpha],\{\alpha\}+\epsilon}_{loc}(\mathbb{R}^{n})\cap\mathcal{L}_{\alpha}(\mathbb{R}^{n})\,\,\,$ with  $\epsilon>0$ arbitrarily small,  $0<\alpha, \beta<\min\{2,n\}$, $a,b, p,q \geq 0$ are real numbers, $k,l\geq 0$ are integers, $[t]$ denotes the integer part of real number $t$ and $\{t\}:=t-[t]$.

For any $u\in C^{[\alpha],\{\alpha\}+\epsilon}_{loc}(\mathbb{R}^{n})\cap\mathcal{L}_{\alpha}(\mathbb{R}^{n})$, the nonlocal operator $(-\Delta)^{\frac{\alpha}{2}}$ ($0<\alpha<2$) is defined by (see \cite{CFY,CLL,DQ1,DQ2,S})
\begin{equation}\label{nonlocal defn}
(-\Delta)^{\frac{\alpha}{2}}u(x)=C_{n,\alpha} \, P.V.\int_{\mathbb{R}^n}\frac{u(x)-u(y)}{|x-y|^{n+\alpha}}dy:=C_{n}\lim_{\varepsilon\rightarrow0}\int_{|y-x|\geq\varepsilon}\frac{u(x)-u(y)}{|x-y|^{n+\alpha}}dy,
\end{equation}
where the function space
\begin{equation}\label{0-1}
\mathcal{L}_{\alpha}(\mathbb{R}^{n}):=\Big\{u: \mathbb{R}^{n}\rightarrow\mathbb{R}\,\big|\,\int_{\mathbb{R}^{n}}\frac{|u(x)|}{1+|x|^{n+\alpha}}dx<\infty\Big\}.
\end{equation}
Throughout this paper, for $m=k\, \text{or}\, l$ and $0<\alpha<2$, we define $(-\Delta)^{m+\frac{\alpha}{2}}f:=(-\Delta)^{m}(-\Delta)^{\frac{\alpha}{2}}f$ for $f\in C^{2m+[\alpha],\{\alpha\}+\epsilon}_{loc}(\mathbb{R}^{n})\cap\mathcal{L}_{\alpha}(\mathbb{R}^{n})$, where $(-\Delta)^{\frac{\alpha}{2}}f$ is defined by definition \eqref{nonlocal defn}. The assumption  $u\in C^{2k+[\alpha],\{\alpha\}+\epsilon}_{loc}(\mathbb{R}^{n})$ and $ v\in C^{2l+[\beta],\{\beta\}+\epsilon}_{loc}(\mathbb{R}^{n})$ with arbitrarily small $\epsilon>0$ guarantees that $(-\Delta)^{\frac{\alpha}{2}}u\in C^{2k}(\mathbb{R}^{n})$ and $(-\Delta)^{\frac{\beta}{2}}v\in C^{2l}(\mathbb{R}^{n})$ (see \cite{CLM,S}), hence $(u,v)$ is a classical solution to  \eqref{HPDE} in the sense that $(-\Delta)^{k+\frac{\alpha}{2}}u$ and $(-\Delta)^{l+\frac{\beta}{2}}v$ are pointwise well-defined and continuous in the whole $\mathbb{R}^{n}$.

When $k=l=1$ and $\alpha=\beta=a=b=0$, system \eqref{HPDE} is the well-known Lane-Emden system
\begin{equation}\label{LMPDE}\\\begin{cases}
-\Delta u(x)=v^p, \,\,\,\,\,\, x\in\mathbb{R}^{n}, \\
-\Delta v(x)=u^q, \,\,\,\,\, x\in\mathbb{R}^{n}.
\end{cases}\end{equation}
The celebrated Lane-Emden conjecture, which states that  the system \eqref{LMPDE} admits no  non-negative non-trivial solutions provided $p>0, q>0$ and $\frac{1}{p+1}+\frac{1}{q+1}>1-\frac{2}{n}$, has been extensively studied and solved for $n\leq 4$ in \cite{M, PQS, S2, Souto, SZ}. Recently, Li and Zhang \cite{LZ} proved the H\'enon-Lane-Emden conjecture for $n=3$.

For $k=l=0$, $0<\alpha=\beta<2$ and $a=b=0$, \eqref{HPDE} is reduced to the fractional Lane-Emden system
\begin{equation}\label{FLMPDE}\\\begin{cases}
(-\Delta)^{\frac{\alpha}{2}} u(x)=v^p, \,\,\,\,\,\, x\in\mathbb{R}^{n}, \\
(-\Delta)^{\frac{\alpha}{2}} v(x)=u^q, \,\,\,\,\, x\in\mathbb{R}^{n}.
\end{cases}\end{equation}
Naturally, one can also ask whether system \eqref{FLMPDE} admits no non-trivial solutions for $p>0, q>0$ and $\frac{1}{p+1}+\frac{1}{q+1}>1-\frac{\alpha}{n}$ or not. Not much work has been done on this question. As far as we know, the best results  were obtained in \cite{QX} via the method of moving planes and \cite{BA} with probabilistic method.

One can see that decay estimates for second order Lane-Emden system and H\'enon-Lane-Emden system  (i.e. Proposition 2.1 in \cite{SZ} and Lemma 3.3 in  \cite{LZ}) are essential in the proof of Lane-Emden conjecture in dimension 3 and 4 and H\'enon-Lane-Emden conjecture in dimension 3. However, the main technique used in \cite{SZ} and \cite{LZ} can not be applied directly to our problem since the fractional Laplacians are nonlocal.

We overcome this difficulty in the fractional order cases (i.e. $k=l=0$) through an interesting observation. The nonlocal average $\int_{R}^{+\infty}\frac{R^{\alpha}}{r(r^{2}-R^{2})^{\frac{\alpha}{2}}}\overline{u}(r)dr$ plays a basic and important role for the nonlocal fractional Laplacians $(-\Delta)^{\frac{\alpha}{2}}$ ($0<\alpha<2$), which is similar to the role of $\overline{u}(r)$ for the Laplacian $-\Delta$. Instead of deriving decay estimates for the spherical average of solutions via ODE method, we will use integral representation formulae to obtain decay estimates for the nonlocal average of solutions to \eqref{HPDE}, which turn out to be enough to deduce Liouville type theorems. The nonlocal average of solutions has been used to obtain super-poly harmonic properties for higher-order fractional equations in \cite{CDQ} and \cite{DLQ}. In this paper, combined with careful analysis, we will show how to use this observation to derive a prior estimates for fractional systems. We believe that techniques developed here will be useful in the study of fractional equations. In the higher order cases $k,l\geq 1$, applying Theorem 1.4 in \cite{CDQ}, we obtain the equivalence between PDE systems \eqref{HPDE} and the corresponding integral systems \eqref{HIE} when $0<2k+\alpha, 2l+\beta <n$. Then we investigate the corresponding integral systems and establish the desired decay estimates by some basic properties of integral equations (see Proposition \ref{2-2-1}).

Our decay estimates for \eqref{HPDE} are as follows.
\begin{thm}\label{Upper}
	Let $(u,v)$ be a non-negative solution of system \eqref{HPDE}. If $p,q\geq 1$ and $pq>1$, then there exists a positive constant $M=M(k,l,\alpha,\beta,a,b,p,q,n)$ such that, for $R>0$,
	\begin{align}\label{UB}
	\int_{R}^{+\infty}\frac{R^{\alpha}}{r(r^{2}-R^{2})^{\frac{\alpha}{2}}}\overline{u}(r)dr&\leq MR^{-\frac{2k+\alpha+a+p(2l+\beta+b)}{pq-1}},\\ \int_{R}^{+\infty}\frac{R^{\beta}}{r(r^{2}-R^{2})^{\frac{\beta}{2}}}\overline{v}(r)dr&\leq MR^{-\frac{2l+\beta+b+q(2k+\alpha+a)}{pq-1}},\nonumber
	\end{align}
	where the spherical average $\overline{u}$ is taken with respect to $0$. Furthermore, for $k,l \geq 1$,  estimate \eqref{UB} can be local
	\begin{align}\label{UB2}
	\overline{u}(R)\leq MR^{-\frac{2k+\alpha+a+p(2l+\beta+b)}{pq-1}},\qquad \overline{v}(R)\leq MR^{-\frac{2l+\beta+b+q(2k+\alpha+a)}{pq-1}},\,\,\,\forall\,\, R>0.
	\end{align}
\end{thm}

One can see that  \eqref{UB} and \eqref{UB2} generalize the crucial estimates  for second order systems(i.e. Proposition 2.1 in \cite{SZ} and Lemma 3.3 in  \cite{LZ}). Note that  \eqref{UB2} is stronger than \eqref{UB},  therefore for $k,l \geq1$, we only need to prove \eqref{UB2}. It is kind of surprising that we derive a local estimate \eqref{UB2} for the nonlocal higher-order fractional H\'enon-Lane-Emden systems.

The proof of Proposition \ref{2-2-1} has its own interest. As an application, we also construct an interesting example on super $\frac{\alpha}{2}$-harmonic functions.

\begin{prop}\label{P1}
	For $0<\alpha<2$, there exists a positive radially symmetric function $u$, such that $(-\Delta)^{\frac{\alpha}{2}}u\geq 0$ in $\mathbb{R}^{n}$, but $u(x)$ is strictly increasing  with respect to $|x|$ for $|x|<1$.
\end{prop}

It is well-known that the spherical average of  a super harmonic function is monotone decreasing.  One may ask  whether $\overline{u}$  decrease for super $\frac{\alpha}{2}$-harmonic function $u$?  In general, Proposition \ref{P1} gives a negative answer. This reveals another essential difference between the nonlocal fractional Laplacians $(-\Delta)^{\frac{\alpha}{2}}$ and the regular Laplacian $-\Delta$.

As an immediate application of Theorem \ref{Upper},  we obtain the following Liouville type theorem.
\begin{thm}\label{thm2}
Let $(u,v)$ be a solution to \eqref{HPDE}. Suppose that $ p,q\geq 1$. Then $(u,v)=(0,0)$ if one of the following three conditions holds,\\
	\emph{(i)} $p=q=1$,\\
	\emph{(ii)} $pq>1$ and $\frac{2k+\alpha+a+p(2l+\beta+b)}{pq-1}>n-2k-\alpha$,\\
	\emph{(iii)} $pq>1$ and $\frac{2l+\beta+b+q(2k+\alpha+a)}{pq-1}>n-2l-\beta$.\\	
\end{thm}

A. Biswas \cite{BA} proved a similar Liouville type theorem for super-solutions of fractional systems with general nonlinearities through a probabilistic approach. We give a direct proof of PDE  method for both fractional and  higher-order fractional systems. Under suitable growth assumptions on $f$ and $g$, our methods can also be applied to the following systems with general nonlinearities.
\begin{equation}\label{GHPDE}\\\begin{cases}
		(-\Delta)^{k+\frac{\alpha}{2}}u(x)=f(x,u,v), \,\,\,\,\,\,\,\, x\in\mathbb{R}^{n}, \\
		(-\Delta)^{l+\frac{\beta}{2}}v(x)=g(x,u,v), \,\,\,\,\, x\in\mathbb{R}^{n},\\
		u(x)\geq 0, v(x)\geq 0,\,\,\, x\in \mathbb{R}^{n}.
\end{cases}\end{equation}
We leave this problem to interested readers.

Next, we investigate the fractional and higher-order fractional  H\'enon-Lane-Emden systems \eqref{HPDE} using other methods. We obtain the following Liouville type theorem for \eqref{HPDE} via the method of scaling spheres in integral form when $2k+\alpha, 2l+\beta <n$ and using Theorem 1.13 in \cite{CDQ} if $\max\{2k+\alpha, 2l+\beta\}\geq n$.
\begin{thm}\label{thm3}
	If $0<2k+\alpha, 2l+\beta <n$, $0< p\leq\frac{n+2k+\alpha+2a}{n-2l-\beta}$, $0< q\leq\frac{n+2l+\beta+2b}{n-2k-\alpha}$ and $(p,q)\not=(\frac{n+2k+\alpha+2a}{n-2l-\beta},\frac{n+2l+\beta+2b}{n-2k-\alpha})$, or if $\max\{2k+\alpha, 2l+\beta\}\geq n$ and $n\geq 2$, then  system \eqref{HPDE} admits no non-trivial solutions.
\end{thm}

The method of scaling spheres developed recently in 2018 in \cite{DQ2} can be successfully applied to various PDE or IE problems  where the method of moving spheres and  the method of moving planes in conjunction with Kelvin transforms do not work(see \cite{DQ2,DQ3,DQ4,DQZ}). Especially, in \cite{DQ2} the authors obtained the sharp Liouville theorems for fractional and higher-order sub-critical Hardy-H\'{e}non-Lane-Emden equations via the method of scaling spheres. Peng \cite{DP} also proved Theorem \ref{thm3} for $k=l=0$, $0<\alpha=\beta<2$ and $k=l\geq 1$, $\alpha=\beta=0$ through similar methods.

For single higher-order fractional equations, Cao, Dai and Qin \cite{CDQ} obtained the Liouville type theorem in $\mathbb{R}^n$ by  establishing super-harmonic properties and the method of scaling spheres in integral form. Zhuo and Li \cite{ZL} proved the Liouville type theorem in $\mathbb{R}^n_+$ with some additional assumptions. Falzy and Wei classified all stable solutions in \cite{FW1} and all finite Morse index solutions in \cite{FW2} via monotonicity formulae. For more literatures on the classification of solutions and Liouville type theorems for various PDE and IE problems please refer to \cite{CD, CDQ2, CL, CL2, DFH, DL, DL2, DQ2, DQ3, DQ4, DQZ} and the references therein.

In the following, we will use $C$ to denote a general positive constant that may depend on $n$, $k$, $l$, $\alpha$, $\beta$, $a$, $b$, $p$, $q$,  $u$ and $v$, whose value may differ from line to line. The spherical average is always taken with respect to the center $0$ for convenience.

\section{Decay estimates and applications}
This section is devoted to prove the decay estimates for solutions of fractional and  higher-order fractional H\'enon-Lane-Emden systems(i.e. Theorem \ref{Upper}). As immediate applications, we will obtain Proposition \ref{P1} and Theorem \ref{thm2}.

\subsection{Proof of Theorem \ref{Upper}}
Applying Theorem 1.4 in \cite{CDQ}, we obtain  the equivalence between PDE systems \eqref{HPDE} and corresponding IE systems.
\begin{lem}\label{HEQ}
	Let $(u,v)$ be a non-negative solution to systems \eqref{HPDE}. Suppose that $0<2k+\alpha, 2l+\beta <n$ and $p,q\geq  0$, then $(u,v)$ also solves the following integral systems
	\begin{equation}\label{HIE}\\\begin{cases}
	u(x)=\int_{\mathbb{R}^n} \frac{R_{2k+\alpha,n}}{|x-y|^{n-2k-\alpha}}|y|^av^p(y)dy, \,\,\,\,\,\,\,\, x\in\mathbb{R}^{n}, \\
	v(x)=\int_{\mathbb{R}^n} \frac{R_{2l+\beta,n}}{|x-y|^{n-2l-\beta}}|y|^bu^q(y)dy, \,\,\,\,\,\,\,\,\, x\in\mathbb{R}^{n},
	\end{cases}\end{equation}
	where the Riesz potential's constants $R_{\gamma,n}:=\frac{\Gamma(\frac{n-\gamma}{2})}{\pi^{\frac{n}{2}}2^{\gamma}\Gamma(\frac{\gamma}{2})}$ for $0<\gamma<n$ (see \cite{Stein}).
\end{lem}

In the following, we will prove the decay estimates \eqref{UB} and \eqref{UB2} separately.

\subsubsection{Fractional cases}
We first prove the key decay estimate \eqref{UB} for $(u, v)$ in the fractional cases $k=l=0$.
From the first equation in \eqref{HPDE},  for arbitrary $R>0$, we conclude that
\begin{equation}\label{2-1}
u(x)=\int_{B_{R}(0)}G^\alpha_R(x,y)|y|^a v^{p}(y)dy+\int_{|y|>R}P^{\alpha}_{R}(x,y)u(y)dy,
\end{equation}
where the Green function for $(-\Delta)^{\frac{\alpha}{2}}$ with $0<\alpha<2$ on $B_R(0)$ is given by
\begin{equation*}
G^\alpha_R(x,y):=\frac{C_{n,\alpha}}{|x-y|^{n-\alpha}}\int_{0}^{\frac{t_{R}}{s_{R}}}\frac{b^{\frac{\alpha}{2}-1}}{(1+b)^{\frac{n}{2}}}db
\,\,\,\,\,\,\,\,\, \text{if} \,\, x,y\in B_{R}(0),
\end{equation*}
with $s_{R}=\frac{|x-y|^{2}}{R^{2}}$, $t_{R}=\left(1-\frac{|x|^{2}}{R^{2}}\right)\left(1-\frac{|y|^{2}}{R^{2}}\right)$, and $G^{\alpha}_{R}(x,y)=0$ if $x$ or $y\in\mathbb{R}^{n}\setminus B_{R}(0)$ (see \cite{K}) and the Poisson kernel $P^{\alpha}_{R}(x,y)$ for $(-\Delta)^{\frac{\alpha}{2}}$ in $B_{R}(0)$ is defined by $P^{\alpha}_{R}(x,y):=0$ for $|y|<R$ and
\begin{equation*}
P^{\alpha}_{R}(x,y):=\frac{\Gamma(\frac{n}{2})}{\pi^{\frac{n}{2}+1}}\sin\frac{\pi\alpha}{2}\left(\frac{R^{2}-|x|^{2}}{|y|^{2}-R^{2}}\right)^{\frac{\alpha}{2}}\frac{1}{|x-y|^{n}},
\end{equation*}
for $|y|>R$ (see \cite{CLM}).

Taking $x=0$ in \eqref{2-1}, we have
\begin{align}\label{2-4}
u(0)&= \int_{B_{R}(0)}\frac{C_{n,\alpha}}{|y|^{n-\alpha}}\left(\int_{0}^{\frac{R^{2}}{|y|^{2}}-1}\frac{b^{\frac{\alpha}{2}-1}}{(1+b)^{\frac{n}{2}}}db\right)|y|^a v^{p}(y)dy\nonumber\\
&\quad+\int_{|y|>R}\frac{C'_{n,\alpha}R^\alpha}{(|y|^2-R^2)^{\frac{\alpha}{2}}|y|^n}u(y)dy\\\nonumber
&= C_0\int_{0}^{R}r^{\alpha+a-1}\left(\int_{0}^{\frac{R^{2}}{r^{2}}-1}\frac{b^{\frac{\alpha}{2}-1}}{(1+b)^{\frac{n}{2}}}db\right)\overline{v^{p}}(r)dr \\
&\quad +C'_0\int_{R}^{+\infty}\frac{R^{\alpha}}{r(r^{2}-R^{2})^{\frac{\alpha}{2}}}\overline{u}(r)dr.\nonumber
\end{align}
Replacing $R$ by $NR$ in \eqref{2-4}, we get
\begin{align}\label{2-5}
u(0)&=C_0\int_{0}^{NR}r^{\alpha+a-1}\left(\int_{0}^{\frac{N^2R^{2}}{r^{2}}-1}\frac{b^{\frac{\alpha}{2}-1}}{(1+b)^{\frac{n}{2}}}db\right)\overline{v^{p}}(r)dr\\
&\qquad+C'_0\int_{NR}^{+\infty}\frac{N^{\alpha}R^{\alpha}}{r(r^{2}-N^2R^{2})^{\frac{\alpha}{2}}}\overline{u}(r)dr.\nonumber
\end{align}
Now, from \eqref{2-4} and \eqref{2-5}, we deduce
\begin{align}\label{2-6}
&\quad\int_{R}^{+\infty}\frac{R^{\alpha}}{r(r^{2}-R^{2})^{\frac{\alpha}{2}}}\overline{u}(r)dr\nonumber\\
&\geq C_1\int_{0}^{R}r^{\alpha+a-1}\left(\int_{\frac{R^{2}}{r^{2}}-1}^{\frac{N^2R^{2}}{r^{2}}-1}\frac{b^{\frac{\alpha}{2}-1}}{(1+b)^{\frac{n}{2}}}db\right)\overline{v}^{p}(r)dr\nonumber
\\&\quad+C_1\int_{R}^{NR}r^{\alpha+a-1}\left(\int_{0}^{\frac{N^2R^{2}}{r^{2}}-1}\frac{b^{\frac{\alpha}{2}-1}}{(1+b)^{\frac{n}{2}}}db\right)\overline{v}^{p}(r)dr \\
&\geq C_1\int_{\frac{R}{2}}^{R}r^{\alpha+a-1}\left(\int_{3}^{\frac{N^2R^{2}}{r^{2}}-1}\frac{b^{\frac{\alpha}{2}-1}}{(1+b)^{\frac{n}{2}}}db\right)\overline{v}^{p}(r)dr\nonumber \\
&\quad+C_1\int_{R}^{NR}r^{\alpha+a-1}\left(\int_{0}^{\frac{N^2R^{2}}{r^{2}}-1}\frac{b^{\frac{\alpha}{2}-1}}{(1+b)^{\frac{n}{2}}}db\right)\overline{v}^{p}(r)dr\nonumber.
\end{align}
Letting $N\rightarrow +\infty$ and by the Lebesgue's dominated convergence theorem, we arrive at
\begin{equation}\label{2-7}
\int_{R}^{+\infty}\frac{R^{\alpha}}{r(r^{2}-R^{2})^{\frac{\alpha}{2}}}\overline{u}(r)dr\geq C \int_{\frac{R}{2}}^{+\infty}r^{\alpha+a-1}\overline{v}^{p}(r)dr.
\end{equation}

Then by H\"{o}lder's inequality, for fixed $t\geq3$, we have
\begin{align*}
&\quad\int_{\frac{R}{2}}^{+\infty}s^{-t}\int_s^{+\infty}\frac{s^{\beta}}{r(r^{2}-s^{2})^{\frac{\beta}{2}}}\overline{v}(r)drds
=\int_{\frac{R}{2}}^{+\infty}\frac{\overline{v}(r)}{r}\int_{\frac{R}{2}}^{r}\frac{s^{\beta-t}}{(r^{2}-s^{2})^{\frac{\beta}{2}}}dsdr\nonumber\\
&\leq C\left(\int_{\frac{R}{2}}^{+\infty}\frac{\overline{v}(r)}{r^{t}}dr+R^{1+\beta-t}\int_{\frac{R}{2}}^{+\infty} \frac{\overline{v}(r)}{r^{1+\beta}}dr\right)\nonumber\\
&\leq C R^{1-t-\frac{\alpha+a}{p}} \left(\int_{\frac{R}{2}}^{+\infty}r^{\alpha+a-1}{\overline{v}}^{p}(r)dr\right)^{\frac{1}{p}}\nonumber.
\end{align*}
From \eqref{2-4}, it is not difficult to verify that $\int_{R}^{+\infty}\frac{R^{\alpha}}{r(r^{2}-R^{2})^{\frac{\alpha}{2}}}\overline{u}(r)dr$ is non-increasing with respect to $R$, and so is $\int_{R}^{+\infty}\frac{R^{\beta}}{r(r^{2}-R^{2})^{\frac{\beta}{2}}}\overline{v}(r)dr$.
Hence, we obtain
\begin{align}\label{2-8}
&\quad\int_{\frac{R}{2}}^{+\infty}r^{\alpha+a-1}\overline{v}^{p}(r)dr\nonumber\\
&\geq CR^{tp-p+\alpha+a} \left(\int_{\frac{R}{2}}^{+\infty}s^{-t}\int_s^{+\infty}\frac{s^{\beta}}{r(r^{2}-s^{2})^{\frac{\beta}{2}}}\overline{v}(r)drds\right)^p\\
&\geq CR^{tp-p+\alpha+a} \left(\int_{\frac{R}{2}}^{R}R^{-t}\int_R^{+\infty}\frac{R^{\beta}}{r(r^{2}-R^{2})^{\frac{\beta}{2}}}\overline{v}(r)drds\right)^p\nonumber\\
&\geq CR^{\alpha+a}\left(\int_R^{+\infty}\frac{R^{\beta}}{r(r^{2}-R^{2})^{\frac{\beta}{2}}}\overline{v}(r)dr\right)^p\nonumber.
\end{align}
We infer from \eqref{2-7} and \eqref{2-8} that
\begin{equation}\label{2-9}
\int_{R}^{+\infty}\frac{R^{\alpha}}{r(r^{2}-R^{2})^{\frac{\alpha}{2}}}\overline{u}(r)dr\geq CR^{\alpha+a}\left(\int_R^{+\infty}\frac{R^{\beta}}{r(r^{2}-R^{2})^{\frac{\beta}{2}}}\overline{v}(r)dr\right)^p.
\end{equation}

Similarly, from the second equation in \eqref{HPDE}, we can deduce
\begin{equation}\label{2-10}
\int_{R}^{+\infty}\frac{R^{\beta}}{r(r^{2}-R^{2})^{\frac{\beta}{2}}}\overline{v}(r)dr\geq CR^{\beta+b}\left(\int_R^{+\infty}\frac{R^{\alpha}}{r(r^{2}-R^{2})^{\frac{\alpha}{2}}}\overline{u}(r)dr\right)^q.
\end{equation}

If $p,q>1$, then \eqref{UB} follows immediately from\eqref{2-9} and \eqref{2-10}.

\subsubsection{Higher-order fractional cases}

We first establish the following basic property of integral equations.
\begin{prop}\label{2-2-1}
	Assume that $u\in C^1(\mathbb{R}^n)$ satisfies the following integral equation
	\begin{equation}\label{22-1}
	u(x)=\int_{\mathbb{R}^n} \frac{f(y)}{|x-y|^{n-\gamma}}dy,
	\end{equation}
	with $2\leq\gamma<n$ and $f\geq 0$.
	
	Then $\overline{u}(r)=\int_{\mathbb{R}^n} \frac{\overline{f}(|y|)}{|x-y|^{n-\gamma}}dy$ and $\overline{u}(r)$ decreases with respect to $r>0$.
\end{prop}

\begin{proof}
	For all $y\in \mathbb{R}^n$, the spherical average 	$\Gamma(r,|y|):=\frac{1}{|\partial B_{r}(0)|}\int_{\partial B_{r}(0)} \frac{1}{|x-y|^{n-\gamma}}d\sigma_x$ depends only on $r=|x|$ and $|y|$, by which we can verify that
	\begin{align}\label{22-2}
	&\quad\overline{u}(r)=\overline{u}(|x|)\nonumber\\\nonumber
	&=\frac{1}{|\partial B_{r}(0)|}\int_{\partial B_{r}(0)}\int_{\mathbb{R}^n} \frac{f(y)}{|x-y|^{n-\gamma}}dyd\sigma_x\\\nonumber
	&=\int_{\mathbb{R}^n}\frac{1}{|\partial B_{r}(0)|}\int_{\partial B_{r}(0)} \frac{f(y)}{|x-y|^{n-\gamma}}d\sigma_xdy\\
	&=\int_{\mathbb{R}^n} f(y)\Gamma(r,|y|)dy\\\nonumber
	&=\int_{\mathbb{R}^n} \overline{f}(|y|)\Gamma(r,|y|)dy\\\nonumber
	&=\frac{1}{|\partial B_{r}(0)|}\int_{\partial B_{r}(0)}\int_{\mathbb{R}^n} \frac{\overline{f}(|y|)}{|x-y|^{n-\gamma}}dy\sigma_x\\\nonumber
	&=\int_{\mathbb{R}^n} \frac{\overline{f}(|y|)}{|x-y|^{n-\gamma}}dy,
	\end{align}
	for any $x\in \mathbb{R}^n$ and $r=|x|$.
	We also used the fact that $\int_{\mathbb{R}^n} \frac{\overline{f}(|y|)}{|x-y|^{n-\gamma}}dy$ depends only on $|x|$.
	
	Taking derivative of \eqref{22-2}, we have
	\begin{equation}\label{22-3}
	\frac{d\overline{u}}{dr}(r)=\frac{x}{|x|}\cdot \nabla \overline{u}(x)=\frac{\gamma-n}{r}\int_{\mathbb{R}^n} \frac{x(x-y)}{|x-y|^{n-\gamma+2}}\overline{f}(|y|)dy,
	\end{equation}
	for all $|x|=r>0$.
	
	Obviously, if $0<R<r$, there holds
	\begin{equation}\label{22-4}
	\int_{\partial B_R(0)}\frac{x(x-y)}{|x-y|^{n-\gamma+2}}d\sigma_y \geq 0.
	\end{equation}
	We claim that, for all $R>r=|x|>0$, if $\gamma>2(=0, <0)$, then
		\begin{equation}\label{22-5}
		\int_{\partial B_R(0)} \frac{x(x-y)}{|x-y|^{n-\gamma+2}}d\sigma_y > 0(=0, <0),
		\end{equation}
	Without loss of generality, we may assume that $r=1$ and $x=e_n$. 
	
	\begin{figure}[ht]
		\includegraphics[scale=0.8]{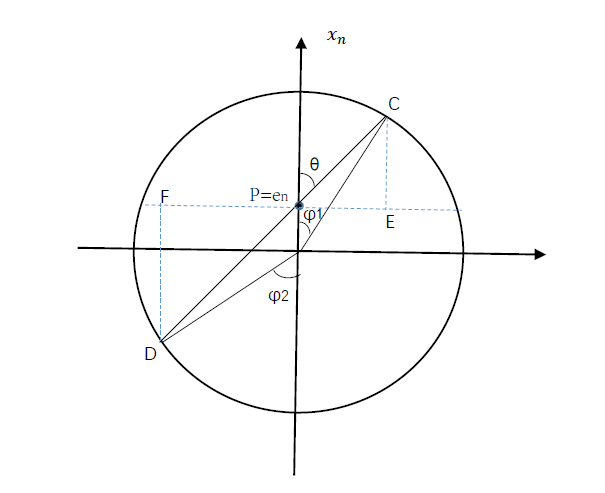} 
		\caption{Spherical coordinate system}
	\end{figure}
	Using the spherical coordinate system with parameter $\theta$ as shown in the figure, we have
	\begin{align}\label{22-6}
	&\int_{\partial B_R(0)} \frac{1-y_n}{|e_n-y|^{n-\gamma+2}}d\sigma_y \nonumber\\
	=&\int_{\partial B_R(0)\cap\{y_n\geq 1\}} \frac{1-y_n}{|e_n-y|^{n-\gamma+2}}d\sigma_y+\int_{\partial B_R(0)\cap\{y_n<1\}} \frac{1-y_n}{|e_n-y|^{n-\gamma+2}}d\sigma_y \\\nonumber
	=&CR\int_0^{\frac{\pi}{2}} -\frac{|PC|\cos\theta}{|PC|^{n-\gamma+2}} (|PC|\sin\theta)^{n-2}\frac{d\varphi_1}{d\theta}(\theta) +\frac{|PD|\cos\theta}{|PD|^{n-\gamma+2}} (|PD|\sin\theta)^{n-2}\frac{d\varphi_2}{d\theta}(\theta) d\theta\\\nonumber
	=&CR\int_0^{\frac{\pi}{2}} -\frac{\cos\theta(\sin\theta)^{n-2}}{|PC|^{3-\gamma}} \frac{d\varphi_1}{d\theta}(\theta)
	+ \frac{\cos\theta(\sin\theta)^{n-2}}{|PD|^{3-\gamma}} \frac{d\varphi_2}{d\theta}(\theta) d\theta.
	\end{align}
	From the law of sines, we have
	\begin{equation}\label{22-7}
	\frac{1}{\sin(\theta-\varphi_1)}=\frac{R}{\sin\theta}=\frac{1}{\sin(\varphi_2-\theta)}.
	\end{equation}
	Denoting $\delta=\theta-\varphi_1=\varphi_2-\theta$, by \eqref{22-7}, we deduce
	\begin{equation}\label{22-8}
	\frac{d\varphi_1}{d\theta}(\theta)=1-\frac{\cos\theta}{R\cos\delta},\qquad \frac{d\varphi_2}{d\theta}(\theta)=1+\frac{\cos\theta}{R\cos\delta}.
	\end{equation}
	Therefore, we obtain
	\begin{align}\label{22-9}
	&\int_{\partial B_R(0)} \frac{1-y_n}{|e_n-y|^{n-\gamma+2}}d\sigma_y\\\nonumber
	=&C\int_0^{\frac{\pi}{2}} \frac{\cos\theta(\sin\theta)^{n-2}(R\cos\delta+\cos\theta)}{|PC|^{3-\gamma}R\cos\delta}\Big[\Big(\frac{|PC|}{|PD|}\Big)^{3-\gamma}-\frac{R\cos\delta-\cos\theta}{R\cos\delta+\cos\theta}\Big]  d\theta.
	\end{align}
	If $\gamma=2$, then by the identity $\frac{1}{\sin\delta}=\frac{R}{\sin\theta}$, one has
	\begin{align}\label{22-10}
	&\frac{|PC|}{|PD|}-\frac{R\cos\delta-\cos\theta}{R\cos\delta+\cos\theta}\nonumber\\\nonumber
	=&\frac{R\cos(\varphi_1)-1}{R\cos(\varphi_2)+1}-\frac{R\cos\delta-\cos\theta}{R\cos\delta+\cos\theta}\\
	=&\frac{R\cos(\theta-\delta)-1}{R\cos(\theta+\delta)+1}-\frac{R\cos\delta-\cos\theta}{R\cos\delta+\cos\theta}\\\nonumber
	=&\frac{2R\cos\delta(\sin^2\theta+\cos^2\theta-1)}{(R\cos(\theta+\delta)+1)(R\cos\delta+\cos\theta)}\\\nonumber
	=&0.
	\end{align}
	Thus, for $\gamma=2$, we have $\int_{\partial B_R(0)} \frac{x(x-y)}{|x-y|^{n-\gamma+2}}d\sigma_y =0$. Since $\frac{|PC|}{|PD|}<1$, we can see that $\Big(\frac{|PC|}{|PD|}\Big)^{3-\gamma}-\frac{R\cos\delta-\cos\theta}{R\cos\delta+\cos\theta}>0$ if $\gamma>2$, and $\Big(\frac{|PC|}{|PD|}\Big)^{3-\gamma}-\frac{R\cos\delta-\cos\theta}{R\cos\delta+\cos\theta}<0$ if $\gamma<2$. This concludes the proof of the claim, and Proposition \ref{2-2-1} follows from \eqref{22-3}, \eqref{22-4} and \eqref{22-5}.
\end{proof}

Now, we are able to derive the decay estimate \eqref{UB2} in the case $k,l\geq1$. By Proposition \ref{2-2-1}, we know that $\overline{u}(r)$ and $\overline{v}(r)$  is decreasing for solution of \eqref{HIE} and satisfies
\begin{equation}\label{22-11}\\\begin{cases}
\overline{u}(|x|)=\int_{\mathbb{R}^n} \frac{R_{2k+\alpha,n}}{|x-y|^{n-2k-\alpha}}|y|^a\overline{v^p}(|y|)dy, \,\,\,\,\,\,\,\, x\in\mathbb{R}^{n}, \\
\overline{v}(|x|)=\int_{\mathbb{R}^n} \frac{R_{2l+\beta,n}}{|x-y|^{n-2l-\beta}}|y|^b\overline{u^q}(|y|)dy, \,\,\,\,\,\,\,\,\, x\in\mathbb{R}^{n}.
\end{cases}\end{equation}
From the first equation in \eqref{22-11}, for $p\geq1$, we have
\begin{eqnarray}\label{22-12}
\overline{u}(R)&=& \int_{\mathbb{R}^{n}}\frac{R_{2k+\alpha,n}}{|x-y|^{n-2k-\alpha}}\overline{|y|^{a}v^p}(|y|)dy
\nonumber\\
 &=& \frac{R_{2k+\alpha,n}}{R^{n-2k-\alpha}}\int_{\mathbb{R}^{n}}\frac{|y|^{a}\overline{v}^{p}(|y|)}{\left|\frac{x}{|x|}-\frac{y}{R}\right|^{n-2k-\alpha}}dy
\\
\nonumber &\geq& \frac{C}{R^{n-2k-\alpha}}\int_{0}^{R}\int_{{\partial}B_{s}(0)}\frac{s^{a}\overline{v}^{p}(s)}{\left|\frac{x}{|x|}-\frac{s\sigma}{R}\right|^{n-2k-\alpha}}d(s\sigma)ds,
\end{eqnarray}
where $R=|x|$ and $\sigma$ is an arbitrary unit vector on $\partial B_{1}(0)$. Observe that
\begin{equation}\label{22-13}
\frac{1}{\left|\frac{x}{|x|}-\frac{s\sigma}{R}\right|^{n-2k-\alpha}}\geq\frac{1}{2^{n-2k-\alpha}},\,\,\,\,\,\,\,\,\,\,\, \forall \,\,s\in[0,R] \,\,\,\text{and} \,\,\, \forall \,\, \sigma\in\partial B_{1}(0).
\end{equation}
Thus we infer from \eqref{22-12} that
\begin{eqnarray}\label{22-14}
\overline{u}(R)
&\geq& \frac{C}{R^{n-2k-\alpha}}\int_{0}^{R}w_{n}s^{n-1+a}\frac{\overline{v}^{p}(s)}{2^{n-2k-\alpha}}ds
\nonumber\\ &\geq& \frac{C}{{R}^{n-2k-\alpha}}\overline{v}^{p}(R)\int_{0}^{R}s^{n-1+a}ds
\\\nonumber &=& CR^{\alpha+2k+a}\overline{v}^{p}(R).
\end{eqnarray}
Similarly, for $q\geq1$, we deduce from the second equation of \eqref{22-11} that
\begin{equation}\label{22-15}
\overline{v}(R)\geq CR^{\beta+2l+b}\overline{u}^{q}(R).
\end{equation}
If $p,q>1$,  then \eqref{UB2} is an easy consequence of \eqref{22-14} and \eqref{22-15}. The proof of Theorem \ref{Upper} is completed.

\subsection{Proof of Proposition \ref{P1} and Theorem \ref{thm2}}
In this subsection, we prove Proposition \ref{P1} and Theorem \ref{thm2}. Proposition \ref{P1} is a direct application of  \eqref{22-5}. In fact, let $f\geq0$ be a smooth radially symmetric function supported in $B_2(0)\setminus B_1(0)$ and $f\not \equiv 0$, then
\begin{align}\label{3-1}
u(x)&=\int_{\mathbb{R}^n} \frac{R_{\alpha,n}f(y)}{|x-y|^{n-\alpha}}dy\\\nonumber
&=\int_{B_2(0)\setminus B_1(0)} \frac{R_{\alpha,n}f(y)}{|x-y|^{n-\alpha}}dy
\end{align}
is the desired example.
One can easily verify that $u$ is radially symmetric and satisfies $(-\Delta)^{\frac{\alpha}{2}}u=f\geq 0$ in $\mathbb{R}^n$. For $0<|x|<1$, there is no singularity in the integration of \eqref{3-1}. Hence, we can take derivative as \eqref{22-3}.   Then \eqref{22-5} implies that $\frac{du}{dr}(r)>0$ for $0<r<1$.

Next, we prove Theorem \ref{thm2}.
If $p=q=1$, then we immediately obtain a contradiction from \eqref{2-9}, \eqref{2-10} and \eqref{22-14}, \eqref{22-15} by taking $R$ large. This proves  \emph{(i)} in Theorem \ref{thm2}.

Since $(u,v)$ satisfies the integral system \eqref{HIE}, we can see that if $(u,v)$ is not identically zero then $u>0, v>0$. Further more,  we can infer from the integral system that
\begin{eqnarray*}
	u(x)&\geq&\int_{|y|\leq\frac{1}{2}}\frac{R_{\alpha,n}}{|x-y|^{n-2k-\alpha}}|y|^av^p(y)dy\\
	\nonumber &\geq& \frac{C}{|x|^{n-2k-\alpha}}\int_{|y|\leq\frac{1}{2}}|y|^av^p(y) dy=:\frac{C}{|x|^{n-2k-\alpha}},
\end{eqnarray*}
for all $|x|\geq1$. Similar argument provides a lower bound for $v$. Therefore, we obtain the following lower bound:
\begin{equation}\label{LB0}
u(x)\geq\frac{C}{|x|^{n-2k-\alpha}},\quad  v(x)\geq\frac{C}{|x|^{n-2l-\beta}},\,\,\,\,\,\,\, \forall \,\,\, |x|\geq1.
\end{equation}
The conclusions \emph{(ii)} and \emph{(iii)} in Theorem \ref{thm2} are direct consequence of the lower bound \eqref{LB0} and Theorem \ref{Upper}. This concludes the proof of Theorem \ref{thm2}.

\section{Proof of Theorem \ref{thm3}}

In this section, we prove Theorem \ref{thm3}. If $\max\{2k+\alpha, 2l+\beta\}\geq n$,  Theorem \ref{thm3} follows from  Theorem 1.13 in \cite{CDQ}. We only need to consider the sub-critical order case $2<2k+\alpha, 2l+\beta <n$, $0\leq p\leq\frac{n+2k+\alpha+2a}{n-2l-\beta}$, $0\leq q\leq\frac{n+2l+\beta+2b}{n-2k-\alpha}$ and $(p,q)\not=(\frac{n+2k+\alpha+2a}{n-2l-\beta},\frac{n+2l+\beta+2b}{n-2k-\alpha})$.

Since we have the integral representation \eqref{HIE} for \eqref{HPDE} by Lemma \ref{HEQ}. We are able to prove Theorem \ref{thm3} by applying the method of scaling spheres in integral form.

Without loss of generality, we assume that $0<p<\frac{n+2k+\alpha+2a}{n-2l-\beta}$ and $0<q\leq\frac{n+2l+\beta+2b}{n-2k-\alpha}$. Given any $\lambda>0$, we define the Kelvin transform of $u$ and $v$ by
    \begin{equation}\label{Kelvin1}
	u_{\lambda}(x)=\left(\frac{\lambda}{|x|}\right)^{n-2k-\alpha}u\left(\frac{\lambda^{2}x}{|x|^{2}}\right),\quad
	v_{\lambda}(x)=\left(\frac{\lambda}{|x|}\right)^{n-2l-\beta}v\left(\frac{\lambda^{2}x}{|x|^{2}}\right),
	\end{equation}
	for arbitrary $x\in\mathbb{R}^{n}\setminus\{0\}$.
	Next, we will carry out the process of scaling spheres with respect to the origin $0\in\mathbb{R}^{n}$.
	
	To this end, let $\lambda>0$ be an arbitrary positive real number and let
	\begin{equation}\label{4-1}
	\omega^{\lambda}_1(x):=u_{\lambda}(x)-u(x),\quad	\omega^{\lambda}_2(x):=v_{\lambda}(x)-v(x),
	\end{equation}
	for any $x\in B_{\lambda}(0)\setminus\{0\}$. Define
	\begin{align}\label{4-3}
	B^{-}_{1,\lambda}:=\{x\in B_{\lambda}(0)\setminus\{0\} \, | \, \omega^{\lambda}_1(x)<0\},\\
	B^{-}_{2,\lambda}:=\{x\in B_{\lambda}(0)\setminus\{0\} \, | \, \omega^{\lambda}_2(x)<0\}.\nonumber
	\end{align}
	The scaling sphere process can be divided into two steps.
	
	\emph{Step 1. Start dilating the sphere from near $\lambda=0$.}
	We will first show that, for $\lambda>0$ sufficiently small,
	\begin{equation}\label{4-2}
	\omega^{\lambda}_1(x),\quad \omega^{\lambda}_2(x)\geq0, \,\,\,\,\,\, \forall \,\, x\in B_{\lambda}(0)\setminus\{0\}.
	\end{equation}
	
	Through direct calculations, for any $x\in\mathbb{R}^{n}$, we get
	\begin{align}\label{4-5}
	u(x)&=C\int_{B_{\lambda}(0)}\frac{|y|^{a}}{|x-y|^{n-2k-\alpha}}v^{p}(y)
	+\frac{|y|^{a}}{\left|\frac{|y|}{\lambda}x-\frac{\lambda}{|y|}y\right|^{n-2k-\alpha}}
	\left(\frac{\lambda}{|y|}\right)^{\tau}v_{\lambda}^{p}(y)dy,\\\nonumber
    u_{\lambda}(x)	&=C\int_{B_{\lambda}(0)}\frac{|y|^{a}}{\left|\frac{|y|}{\lambda}x-\frac{\lambda}{|y|}y\right|^{n-2k-\alpha}}v^{p}(y)+\frac{|y|^{a}}{|x-y|^{n-2k-\alpha}}\left(\frac{\lambda}{|y|}\right)^{\tau}v_{\lambda}^{p}(y)dy,
	\end{align}
	where $\tau:=n+2k+\alpha+2a-p(n-2l-\beta)>0$.
	
	Therefore, for any $x\in B_{1,\lambda}^{-}$, we obtain
	\begin{align}\label{4-8}
	&\quad\omega^{\lambda}_1(x)=u_{\lambda}(x)-u(x) 	\nonumber\\
 &=C\int_{B_{\lambda}(0)}\Bigg(\frac{|y|^{a}}{|x-y|^{n-2k-\alpha}}-\frac{|y|^{a}}{\left|\frac{|y|}{\lambda}x-\frac{\lambda}{|y|}y\right|^{n-2k-\alpha}}\Bigg) \left(\left(\frac{\lambda}{|y|}\right)^{\tau}v_{\lambda}^{p}(y)-v^{p}(y)\right)dy\\
	\nonumber &> C\int_{B_{2,\lambda}^{-}}\Bigg(\frac{|y|^{a}}{|x-y|^{n-2k-\alpha}}-\frac{|y|^{a}}{\left|\frac{|y|}{\lambda}x-\frac{\lambda}{|y|}y\right|^{n-2k-\alpha}}\Bigg)
	\max\left\{v^{p-1}(y),v_{\lambda}^{p-1}(y)\right\}\omega^{\lambda}_2(y)dy\\
	\nonumber &\geq C\int_{B_{2,\lambda}^{-}}\frac{|y|^{a}}{|x-y|^{n-2k-\alpha}}\max\left\{v^{p-1}(y),v_{\lambda}^{p-1}(y)\right\}\omega^{\lambda}_2(y)dy.
	\end{align}
	Similarly, for any $x\in B_{2,\lambda}^{-}$, we have
	\begin{eqnarray}\label{4-9}
	\omega^{\lambda}_2(x) \geq C\int_{B_{1,\lambda}^{-}}\frac{|y|^{b}}{|x-y|^{n-2l-\beta}}\max\left\{u^{q-1}(y),u_{\lambda}^{q-1}(y)\right\}\omega^{\lambda}_1(y)dy.
	\end{eqnarray}
	By Hardy-Littlewood-Sobolev inequality,  we derive from \eqref{4-8} and \eqref{4-9} that
	\begin{eqnarray*}\label{4-10}
	\|\omega^{\lambda}_1\|_{L^{t}(B^{-}_{1,\lambda})}&\leq& C\left\||x|^{a}\max\left\{v^{p-1},v_{\lambda}^{p-1}\right\}\omega^{\lambda}_2\right\|_{L^{\frac{nt}{n+(2k+\alpha)t}}(B^{-}_{2,\lambda})}\\
	\nonumber &\leq& C\left\||x|^{a}\max\left\{v^{p-1},v_{\lambda}^{p-1}\right\}\right\|_{L^{\frac{n}{2k+\alpha}}(B^{-}_{2,\lambda})}\|\omega^{\lambda}_2\|_{L^{t}(B^{-}_{2,\lambda})}
	\end{eqnarray*}
	\begin{eqnarray*}\label{4-11}
	\|\omega^{\lambda}_2\|_{L^{t}(B^{-}_{2,\lambda})}&\leq& C\left\||x|^{b}\max\left\{u^{q-1},u_{\lambda}^{q-1}\right\}\omega^{\lambda}_1\right\|_{L^{\frac{nt}{n+(2l+\beta)t}}(B^{-}_{1,\lambda})}\\
	\nonumber &\leq& C\left\||x|^{b}\max\left\{u^{q-1},u_{\lambda}^{q-1}\right\}\right\|_{L^{\frac{n}{2l+\beta}}(B^{-}_{1,\lambda})}\|\omega^{\lambda}_1\|_{L^{t}(B^{-}_{1,\lambda})},
	\end{eqnarray*}
	for any $\max\{\frac{n}{n-2k-\alpha}, \frac{n}{n-2l-\beta}\}<t<\infty$.

	Therefore, we arrive at
	\begin{align}\label{4-12}
	\|\omega^{\lambda}_1\|_{L^{t}(B^{-}_{1,\lambda})}\leq C_{\lambda, u, v}\|\omega^{\lambda}_1\|_{L^{t}(B^{-}_{1,\lambda})},\\\nonumber
	\|\omega^{\lambda}_2\|_{L^{t}(B^{-}_{2,\lambda})}\leq C_{\lambda, u, v}\|\omega^{\lambda}_2\|_{L^{t}(B^{-}_{2,\lambda})},
	\end{align}
	where $$C_{\lambda, u, v}:=C\left\||x|^{a}\max\left\{v^{p-1},v_{\lambda}^{p-1}\right\}\right\|_{L^{\frac{n}{2k+\alpha}}(B^{-}_{2,\lambda})}\left\||x|^{b}\max\left\{u^{q-1},u_{\lambda}^{q-1}\right\}\right\|_{L^{\frac{n}{2l+\beta}}(B^{-}_{1,\lambda})}.$$
	
	Since \eqref{LB0} implies that
	\begin{equation*}
	\inf_{x\in B_{\lambda}(0)\setminus\{0\}}u_{\lambda}(x)\geq C,\,\inf_{x\in B_{\lambda}(0)\setminus\{0\}}v_{\lambda}(x)\geq C,
	\end{equation*}
	for any $\lambda\leq1$, we have for $0<p<1$ and $0<q<1$,
	\begin{equation*}
	\sup_{x\in B_{\lambda}(0)\setminus\{0\}}u_{\lambda}^{q-1}(x)\leq C,\,\sup_{x\in B_{\lambda}(0)\setminus\{0\}}v_{\lambda}^{p-1}(x)\leq C.
	\end{equation*}
	
	Thus, there exists a $\epsilon_{0}>0$ small enough such that
	\begin{equation}\label{4-13}
	C_{\lambda, u, v}\leq\frac{1}{2}
	\end{equation}
	for all $0<\lambda\leq\epsilon_{0}$.\\
	 Thus, \eqref{4-12} implies
	\begin{equation}\label{4-14}
	\|\omega^{\lambda}_1\|_{L^{t}(B^{-}_{1,\lambda})}=\|\omega^{\lambda}_2\|_{L^{t}(B^{-}_{2,\lambda})}=0, \qquad \forall \,\, 0<\lambda\leq\epsilon_{0},
	\end{equation}
	which means $B^{-}_{1,\lambda}=B^{-}_{2,\lambda}=\emptyset$. Therefore, \eqref{4-2} holds and Step 1 is completed.
	
	\emph{Step 2. Dilate the sphere $S_{\lambda}$ outward until $\lambda=+\infty$. }
	
	In this step, we will dilate the sphere $S_{\lambda}$ outward until $\lambda=+\infty$ to derive lower bound estimates on $u$ and $v$. Step 1 provides us a start point to dilate the sphere $S_{\lambda}$ from place near $\lambda=0$. Now we dilate the sphere $S_{\lambda}$ outward as long as \eqref{4-2} holds. Let
	\begin{equation}\label{4-15}
	\lambda_{0}:=\sup\{\lambda>0\,|\, \omega^{\mu}_1\geq0,\,\omega^{\mu}_2\geq0 \,\, in \,\, B_{\mu}(0)\setminus\{0\}, \,\, \forall \, 0<\mu\leq\lambda\}\in(0,+\infty].
	\end{equation}
Hence, one has
	\begin{equation}\label{4-16}
	\omega^{\lambda_{0}}_1(x)\geq0,\quad \omega^{\lambda_{0}}_2(x)\geq0, \quad\quad \forall \,\, x\in B_{\lambda_{0}}(0)\setminus\{0\}.
	\end{equation}
	Suppose $0<\lambda_{0}<+\infty$. We will first show that
	\begin{equation}\label{4-17}
	\omega^{\lambda_{0}}_1(x)>0, \quad \omega^{\lambda_{0}}_2(x)>0, \,\,\,\,\,\, \forall \, x\in B_{\lambda_{0}}(0)\setminus\{0\}.
	\end{equation}
	Indeed, if we assume
	\begin{equation}\label{4-18}
	\omega^{\lambda}_1(x^0)=0, \qquad \text{for some} \,\, x^0\in B_{\lambda_{0}}(0)\setminus\{0\},
	\end{equation}
	then by the second equality in \eqref{4-8} and \eqref{4-16}, we arrive at
	\begin{align}\label{4-19}
	&\quad 0=\omega^{\lambda_{0}}(x^0)=u_{\lambda_{0}}(x^0)-u(x)\\
	\nonumber &\geq C\int_{B_{\lambda_{0}}(0)}\Bigg(\frac{|y|^{a}}{|x^0-y|^{n-2k-\alpha}}-\frac{|y|^{a}}{\left|\frac{|y|}{\lambda_{0}}x^0-\frac{\lambda_{0}}{|y|}y\right|^{n-2k-\alpha}}\Bigg) \left(\left(\frac{\lambda_{0}}{|y|}\right)^{\tau}-1\right)v^{p}(y)dy>0,
	\end{align}
	 which is absurd. Thus, we must have $\omega^{\lambda_{0}}_1(x)>0$. Similarly, one can prove that $ \omega^{\lambda_{0}}_2(x)>0$  and \eqref{4-17} holds true. Furthermore, by  continuity, on a fixed small ball $B_{\delta}(x^{0})\subset B_{\lambda_{0}}(0)\setminus\{0\}$,  we have $\omega^{\lambda_{0}}_1(x)>c,\quad \omega^{\lambda_{0}}_2(x)>c$ for some constant $c>0$. Then, we can derive from \eqref{4-8} and \eqref{4-9} that there exists a $0<\eta<\lambda_{0}$ small enough such that, for any $x\in \overline{B_{\eta}(0)}\setminus\{0\}$,
	\begin{equation}\label{4-20}
	\omega^{\lambda_{0}}_1(x),\quad \omega^{\lambda_{0}}_2(x)>c_0>0.
	\end{equation}
	
	We fix $0<r_{0}<\frac{1}{2}\lambda_{0}$ small enough, such that
	\begin{equation}\label{4-21}
	C\left\||x|^{a}\max\left\{v^{p-1},v_{\lambda}^{p-1}\right\}\right\|_{L^{\frac{n}{2k+\alpha}}(A_{\lambda_{0}+r_{0},2r_{0}})}\left\||x|^{b}\max\left\{u^{q-1},u_{\lambda}^{q-1}\right\}\right\|_{L^{\frac{n}{2l+\beta}}(A_{\lambda_{0}+r_{0},2r_{0}})}\leq\frac{1}{2}
	\end{equation}
	for any $\lambda\in[\lambda_{0},\lambda_{0}+r_{0}]$, where $C$ is the constant in \eqref{4-12} and the narrow region
	\begin{equation}\label{4-22}
	A_{\lambda_{0}+r_{0},2r_{0}}:=\left\{x\in B_{\lambda_{0}+r_{0}}(0)\,|\,|x|>\lambda_{0}-r_{0}\right\}.
	\end{equation}
	
	We infer from \eqref{4-17} and \eqref{4-20} that
	\begin{equation}\label{4-24}
	m_{1}:=\inf_{x\in\overline{B_{\lambda_{0}-r_{0}}(0)}\setminus\{0\}}\omega^{\lambda_{0}}_1(x)>0,\quad m_{2}:=\inf_{x\in\overline{B_{\lambda_{0}-r_{0}}(0)}\setminus\{0\}}\omega^{\lambda_{0}}_2(x)>0.
	\end{equation}
	Note that \eqref{4-24} is equivalent to
	\begin{equation}\label{4-24-1}
	|x|^{n-2k-\alpha}u(x)-\lambda_{0}^{n-2k-\alpha}u(x^{\lambda_{0}})\geq m_{1}\lambda_{0}^{n-2k-\alpha}, \,\,\,\,\,\,\,\,\, \forall \, |x|\geq\frac{\lambda_{0}^{2}}{\lambda_{0}-r_{0}}.
	\end{equation}
	Since $u$ is uniformly continuous on $\overline{B_{4\lambda_{0}}(0)}$, we infer from \eqref{4-24-1} that there exists a $0<\varepsilon_{1}<r_{0}$ sufficiently small, such that, for any $\lambda\in[\lambda_{0},\lambda_{0}+\varepsilon_{1}]$,
	\begin{equation}\label{4-24-2}
	|x|^{n-2k-\alpha}u(x)-\lambda^{n-2k-\alpha}u(x^{\lambda})\geq \frac{m_{1}}{2}\lambda^{n-2k-\alpha}, \,\,\,\,\,\,\,\,\, \forall \, |x|\geq\frac{\lambda^{2}}{\lambda_{0}-r_{0}},
	\end{equation}
	which is equivalent to
	\begin{equation}\label{4-25}
	\omega^{\lambda}_1(x)\geq\frac{m_{1}}{2}>0, \quad \omega^{\lambda}_2(x)\geq\frac{m_{1}}{2}>0,\,\,\,\,\,\, \forall \, x\in\overline{B_{\lambda_{0}-r_{0}}(0)}\setminus\{0\}.
	\end{equation}
	Similar argument indicates that there exists a $0<\varepsilon_{2}<r_{0}$ sufficiently small, such that, for any $\lambda\in[\lambda_{0},\lambda_{0}+\varepsilon_{2}]$,
	\begin{equation}\label{4-25-1}
	\omega^{\lambda}_2(x)\geq\frac{m_{2}}{2}>0,\,\,\,\,\,\, \forall \, x\in\overline{B_{\lambda_{0}-r_{0}}(0)}\setminus\{0\}.
	\end{equation}
	Set $\varepsilon_3=\min\{\varepsilon_{1},\varepsilon_{2}\}$. By \eqref{4-25} and \eqref{4-25-1}, we know that for any $\lambda\in[\lambda_{0},\lambda_{0}+\varepsilon_{3}]$,
	\begin{equation}\label{4-26}
	B_{1,\lambda}^{-}\cup B_{2,\lambda}^{-} \subset A_{\lambda_{0}+r_{0},2r_{0}}.
	\end{equation}
	Hence, \eqref{4-12}, estimates \eqref{4-21} and \eqref{4-26} yield
	\begin{equation}\label{4-27}
	\|\omega^{\lambda}_1\|_{L^{t}(B^{-}_{1,\lambda})}=\|\omega^{\lambda}_2\|_{L^{t}(B^{-}_{2,\lambda})}=0, \qquad \forall \,\, 0<\lambda\leq\epsilon_{0}
	\end{equation}
	Therefore, for any $\lambda\in[\lambda_{0},\lambda_{0}+\varepsilon_{3}]$, we deduce from \eqref{4-27} that, $B^{-}_{1,\lambda}=B^{-}_{2,\lambda}=\emptyset$. Thus, for any $\lambda\in[\lambda_{0},\lambda_{0}+\varepsilon_{3}]$, we conclude that
	\begin{equation}\label{4-28}
	\omega^{\lambda}_1(x)\geq0, \,\,\,\,\,\,\omega^{\lambda}_2(x)\geq0, \,\,\,\,\,\,\, \forall \,\, x\in B_{\lambda}(0)\setminus\{0\},
	\end{equation}
	which contradicts with the definition \eqref{4-15} of $\lambda_{0}$. Thus we must have $\lambda_{0}=+\infty$, that is,
	\begin{equation}\label{4-29}
	u(x)\geq\left(\frac{\lambda}{|x|}\right)^{n-2k-\alpha}u\left(\frac{\lambda^{2}x}{|x|^{2}}\right),\quad
	v(x)\geq\left(\frac{\lambda}{|x|}\right)^{n-2l-\beta}u\left(\frac{\lambda^{2}x}{|x|^{2}}\right),
	\end{equation}
	for all $|x|\geq\lambda$ and $0<\lambda<+\infty$.
	For arbitrary $|x|\geq1$, let $\lambda:=\sqrt{|x|}$. Then \eqref{4-29} yields the following lower bound estimate:
	\begin{align}\label{4-30}
	u(x)&\geq\left(\min_{x\in S_{1}}u(x)\right)\frac{1}{|x|^{\frac{n-2k-\alpha}{2}}}:=\frac{C_{1}}{|x|^{\frac{n-2k-\alpha}{2}}},\\
	v(x)&\geq\left(\min_{x\in S_{1}}v(x)\right)\frac{1}{|x|^{\frac{n-2l-\beta}{2}}}:=\frac{C_{2}}{|x|^{\frac{n-2l-\beta}{2}}}, \quad\quad \forall \,\, |x|\geq1\nonumber.
	\end{align}
	
	Denoting $\mu_{u,0}:=\frac{n-2k-\alpha}{2}, \quad\mu_{v,0}:=\frac{n-2l-\beta}{2} $ , we infer from the first integral equation in \eqref{HIE} and \eqref{4-30} that for $|x|\geq1$,
	\begin{equation*}
	u(x)\geq C\int_{2|x|\leq|y|\leq4|x|}\frac{1}{|x-y|^{n-2k-\alpha}|y|^{p\mu_{v,0}-a}}dy
	\geq \frac{C_{1}}{|x|^{\mu_{u,1}}},
	\end{equation*}
	where $\mu_{u,1}:=p\mu_{v,0}-(a+2k+\alpha)$. Similarly, we obtain that $v\geq \frac{C}{|x|^{\mu_{v,1}}} $ with $\mu_{v,1}=q\mu_{u,0}-(b+2l+\beta)$.
	For $i=0,1,2,\cdots$, we define
	\begin{equation*}
	\mu_{u,i+1}:=p\mu_{v,i}-(a+2k+\alpha),\quad \mu_{v,i+1}:=q\mu_{u,i}-(b+2l+\beta).
	\end{equation*}
	Then, we have
	\begin{align}\label{4-31}
	\mu_{u,i+2}:=pq\mu_{u,i}-p(b+2l+\beta)-(a+2k+\alpha),\\
	\mu_{v,i+2}:=pq\mu_{v,i}-q(a+2k+\alpha)-(b+2l+\beta)\nonumber.
	\end{align}
	If $pq=1$, then obviously $\mu_{u,2j}, \mu_{v,2j}\rightarrow -\infty$ as $j\rightarrow +\infty$. For $pq\not=1$, noticing that  $\frac{n-2k-\alpha}{2}-\frac{p(n+2l+\beta)+(a+2k+\alpha)}{pq-1}<0$ and $\frac{n-2l-\beta}{2}-\frac{q(a+2k+\alpha)+(b+2l+\beta)}{pq-1}<0$, one can check that
	for $pq>1$,  $\mu_{u,2j}, \mu_{v,2j}\rightarrow -\infty$ as $j\rightarrow +\infty$, and for $pq<1$
	\begin{align*}
	\mu_{u,2j}&\rightarrow \frac{p(b+2l+\beta)+(a+2k+\alpha)}{pq-1}<0,\\
	\mu_{v,2j}&\rightarrow \frac{q(a+2k+\alpha)+(b+2l+\beta)}{pq-1}<0,
	\end{align*}
	as $ j\rightarrow +\infty$.\\
	
In all cases, it is obvious that the lower bound for $|x|>1$ and $j$ large
$$u(x)\geq \frac{C}{|x|^{\mu_{u,2j}}},\quad v(x)\geq \frac{C}{|x|^{\mu_{v,2j}}},$$
contradicts with the decay estimate \eqref{UB}. This concludes the proof of Theorem \ref{thm3}.
\bibliographystyle{amsplain}

\begin{thebibliography}{99}
	\bibitem{BA} A. Biswas, Liouville type results for systems of equations involving fractional
		Laplacian in exterior domains, \emph{Nonlinearity}, \textbf{32} (2019), no.6, 2246-2268.
		
	\bibitem{CD} D. Cao and W. Dai, Classification of nonnegative solutions to a bi-harmonic equation with Hartree type nonlinearity, \emph{Proc. Roy. Soc. Edinburgh Sect. A}, \textbf{149} (2019), 979-994.
	
	\bibitem{CDQ} D. Cao, W. Dai and G.Qin,  Super poly-harmonic properties, Liouville theorems and classification of nonnegative solutions to equations involving higher-order fractional Laplacians, preprint, arXiv: 1905.04300.
	
	\bibitem{CDQ2} W. Chen, W. Dai and G. Qin, Liouville type theorems, a priori estimates and existence of solutions for critical order Hardy-H\'{e}non equations in $\mathbb{R}^n$, preprint, arXiv: 1808.06609.
	
	\bibitem{CL} W. Chen and C. Li, Classification of solutions of some nonlinear elliptic equations, \emph{Duke Math. J.}, \textbf{63} (1991), no. 3, 615-622.
	
	\bibitem{CL2} W. Chen and C. Li, Super polyharmonic property of solutions for PDE systems and its applications, \emph{Comm. Pure Appl. Anal.}, \textbf{12} (2013), 2497-2514.
	
	
	\bibitem{CLM} W. Chen, Y. Li and P. Ma, \emph{ The Fractional Laplacian}, World Scientific Publishing Co. Pte. Ltd., 2020, 344 pp, https://doi.org/10.1142/10550.
	
	\bibitem{CFY} W. Chen, Y. Fang and R. Yang,  Liouville theorems involving the fractional Laplacian on a half space, \emph{Adv. Math.}, \textbf{274} (2015), 167-198.
	
	\bibitem{CLL} W. Chen, C. Li and Y. Li, A direct method of moving planes for the fractional Laplacian, \emph{Adv. Math.}, \textbf{308} (2017), 404-437.
	
	\bibitem{CS} L. Caffarelli and L. Silvestre, An extension problem related to the fractional Laplacian, \emph{Comm. PDE.}, \textbf{32} (2007), 1245-1260.
	
	\bibitem{DFH} W. Dai, Y. Fang J. Huang, Y. Qin and B. Wang, Regularity and classification of solutions to static Hartree equations involving fractional Laplacians, \emph{Discrete Contin. Dyn. Syst. - A}, \textbf{39} (2019), no. 3, 1389-1403.
	
	\bibitem{DL} W. Dai and Z. Liu, Classification of positive solutions to a system of Hardy-Sobolev type equations, \emph{Acta Mathematica Scientia}, \textbf{37} (2017), no. 5, 1415-1436.
		
	\bibitem{DL2} W. Dai and Z. Liu,  Classification of nonnegative solutions to static Schr\"{o}dinger-Hartree and Schr\"{o}dinger-Maxwell equations with combined nonlinearities, \emph{Calc. Var. Partial Differential Equations}, \textbf{58} (2019), no. 4: 156, https://doi.org/10.1007/s00526-019-1595-z.
	
	\bibitem{DLQ} W. Dai, Z. Liu and G. Qin, Classification of nonnegative solutions to static Schr\"{o}dinger-Hartree-Maxwell type equations, preprint, arXiv: 1909.00492.
	
	\bibitem{DQ1} W. Dai and G. Qin, Classification of nonnegative classical solutions to third-order equations, \emph{Adv. Math.}, \textbf{328} (2018), 822-857.
	
	\bibitem{DQ2} W. Dai and G. Qin, Liouville type theorems for fractional and higher order H\'{e}non-Hardy type equations via the method of scaling spheres, preprint, arXiv: 1810.02752.
	
	\bibitem{DQ3} W. Dai and G. Qin, Liouville type theorem for critical order H\'{e}non-Lane-Emden type equations on a half space and its applications, preprint, arXiv: 1811.00881.
	
	\bibitem{DQ4} W. Dai and G. Qin, Liouville type theorems for elliptic equations with Dirichlet conditions in exterior domains,  \emph{Journal of Differential Equations}, \textbf{269} (2020), 7231-7252.
	
	\bibitem{DQ5} W. Dai and G. Qin, Liouville type theorems for Hardy-Henon equations with concave nonlinearities, \emph{Math. Nachr.}, \textbf{293}(2020), no.6, 1084-1093. DOI: 10.1002/mana.201800532.
	
	\bibitem{DQZ} W. Dai, G. Qin and Y. Zhang, Liouville type theorem for higher order H\'{e}non equations on a half space, \emph{Nonlinear Analysis}, \textbf{183} (2019), 284-302.
	

	\bibitem{FW1} M. Fazly and J. Wei,  On stable solutions of the fractional H\'enon-Lane-Emden equation,  \emph{Commun. Contemp. Math.}, \textbf{18} (2016), no.5, 24pp.
	
	\bibitem{FW2} M. Fazly and J. Wei, On finite Morse index solutions of higher order fractional Lane-Emden equations,  \emph{Amer. J. Math.}, \textbf{139} (2017), no.2, 433–460.
	
	\bibitem{K} T. Kulczycki, Properties of Green function of symmetric stable processes, \emph{Probability and Mathematical Statistics}, \textbf{17} (1997), 339-364.
	
	\bibitem{LZ} K.Li and Z.Zhang, Proof of the H\'{e}non-Lane-Emden conjecture in $\mathbb{R}^{3}$, \emph{Journal of Differential Equations}, \textbf{266} (2017), 202-226.
	
	\bibitem{M} E. Mitidieri, Nonexistence of positive solutions of semilinear elliptic systems in $\mathbb{R}^{N}$, \emph{Differential Integral Equations}, \textbf{9} (1996), 465-479.
	
	\bibitem{DP} S. Peng, Liouville theorems for fractional and higher order H\'enon-Hardy systems on $\mathbb{R}^n$, to appear in\emph{Complex Var. Elliptic Equ.}, 2020, 25pp, DOI:10.1080/17476933.2020.1783661.
	
	\bibitem{PQS} P. Pol\'{a}\v{c}ik, P. Quittner and P. Souplet, Singularity and decay estimates in superlinear problems via Liouville-type theorems. Part I: Elliptic systems, \emph{Duke Math. J.}, \textbf{139} (2007), 555-579.
	
	\bibitem{QX} A. Quaas and A.Xia, A Liouville type theorem for Lane-Emden systems involving the fractional Laplacian, \emph{Nonlinerity}, \textbf{29} (2016), 2279-2297.
	
	\bibitem{S} L. Silvestre, Regularity of the obstacle problem for a fractional power of the Laplace operator, \emph{Comm. Pure Appl. Math.}, \textbf{60} (2007), 67-112.
	
	\bibitem{S2} P. Souplet, The proof of the Lane-Emden conjecture in four space dimensions, \emph{Adv. Math.}, \textbf{221} (2009), no. 5, 1409-1427.
	
	\bibitem{Stein} E. M. Stein, \emph{Singular integrals and differentiability properties of functions}, Princeton Landmarks in Mathematics, Princeton University Press, Princeton, New Jersey, 1970.
	
	\bibitem{Souto} M. A. S. Souto, A priori estimates and existence of positive solutions of non-linear cooperative elliptic systems, \emph{Differential Integral Equations}, \textbf{8} (1995), 1245-1258.
	
	\bibitem{SZ} J. Serrin and H. Zou,  Non-existence of positive solutions of Lane-Emden systems, \emph{Differential Integral Equations}, \textbf{9} (1996), no. 4, 635-653.
	
	\bibitem{ZL} R. Zhuo and Y. Li, A Liouville theorem for the higher-order fractional Laplacian,  \emph{Commun. Contemp. Math.}, \textbf{21} (2019), no.2, 19pp.

\end{thebibliography}

\end{document}